\documentclass[letterpaper, 10 pt, conference]{ieeeconf}  

\IEEEoverridecommandlockouts                              
\usepackage{graphics} 
\usepackage{amsmath} 
\usepackage{amssymb}  
\usepackage{color} 
\usepackage{graphicx}
\newtheorem{prop}{Proposition}

\DeclareMathOperator{\e}{\mathrm{e}}

\title{\LARGE \bf
Output Corridor Control via Design of Impulsive Goodwin's Oscillator
}

\author{Alexander Medvedev$^{1}$, Anton V. Proskurnikov$^{2}$, and Zhanybai T. Zhusubaliyev$^{3,4}$
\thanks{* AM was partially supported the Swedish Research Council under grant 2019-04451. ZhZh was partially supported  by the grant 14-22 of the Osh State University.}
\thanks{$^{1}$Department of Information Technology, 
        Uppsala University, SE-752 37 Uppsala, Sweden
        [{\tt\small alexander.medvedev@it.uu.se}]}%
\thanks{$^{2}$Department of Electronics and Telecommunications, Politecnico di Torino, Turin, Italy, 10129 [{\tt\small anton.p.1982@ieee.org}]}%
\thanks{$^{3}$Department of Computer Science, International Scientific Laboratory for
Dynamics of Non-Smooth Systems, Southwest State University, Kursk, Russia
        [{\tt\small zhanybai@hotmail.com}]}%
\thanks{$^{4}$Faculty of Mathematics and Information Technology, Osh State University, Lenin st. 331, 723500, Osh, Kyrgyzstan        }%
}

\begin{document}

\maketitle
\thispagestyle{empty}
\pagestyle{empty}

\begin{abstract}
In the Impulsive Goodwin's oscillator (IGO), a continuous positive linear time-invariant (LTI) plant is controlled by an amplitude- and frequency-modulated feedback into an oscillating solution. Self-sustained oscillations in the IGO model have been extensively used to portray periodic rhythms in endocrine systems,  whereas the potential of the concept as a controller design approach still remains mainly unexplored. This paper proposes an algorithm to  design the feedback of the IGO so that the output of the continuous plant is kept (at stationary conditions) within a pre-defined corridor, i.e. within a bounded interval of values. The presented framework covers single-input single-output LTI plants as well as positive Wiener and Hammerstein models that often appear in process and biomedical control. A potential application of the developed impulsive control approach to a minimal Wiener model of pharmacokinetics and pharmacodynamics of a muscle relaxant used in general anesthesia is discussed.

\end{abstract}

\section{INTRODUCTION}
Governing the output of a dynamical plant to a given set point is by far the most frequently treated problem in control engineering.  Since exact stabilization to a point is impossible in practice due to the effect of disturbances and model uncertainty, it is customary to specify a range  of values that  the controlled plant output is allowed to evolve within.

With respect to tracking of a time-varying reference signal by means of Model Predictive Control with zero-order hold, the output corridor control problem is considered in e.g. \cite{RCS22}. Solving the output corridor problem with a simple controller structure is seldom addressed even for a linear time-invariant (LTI) plant without uncertainty. 

In contrast with a conventional continuous- or discrete-time feedback, event-based control  \cite{HJT12} suggests that a control action is taken when it is necessary to fulfill the control objective. This principle is well in line with e.g. biological control mechanisms that seek to minimize the energy and communication load \cite{WTT10} inflicted by the controller while maintaining homeostasis.  An impulsive controller presents then an attractive option since it combines the principle of event-based control with minimal interaction between the controller and the actuator. A downside of it is that the closed-loop dynamics of impulsive control systems are highly nonlinear and non-smooth, thus requiring a combination of analytical and computational tools to be used in controller design.

The Impulsive Goodwin's oscillator (IGO) \cite{MCS06,Aut09} is a mathematical model that is devised to portray the pulsatile  regulation featured by many endocrine feedback systems. Ostensibly, the goal of endocrine regulation is to keep the concentrations of the involved hormones within certain physiologically beneficial bounds. Impulsive control is utilized in the human organism when it comes  to regulating e.g. the sex hormones, growth hormone, and stress hormone.   Due to the biochemical nature of the considered system, where the state variables correspond to hormone concentrations, the controlled plant is assumed to be positive. This kind of models are typical to chemical and biological systems, e.g. in analysis control of bioreactors or population dynamics.
 
Most of the literature devoted to the IGO covers analysis of the dynamical behaviors arising in autonomous \cite{ZCM12b} and forced \cite{MPZh18} model operation, as well as the complex phenomena due to the introduction of point-wise  \cite{CMM14} or distributed time delay \cite{CM16}. More recently, the problem of IGO design to admit a desired (stable) periodic solution is addressed in \cite{MPZh23,MPZh23a}, primarily  with respect to (discrete) dosing applications. 

The present paper takes further the IGO design approach to guarantee that the obtained periodic solution results in the output of the continuous part of the model being constrained to a pre-defined (positive) interval of values and, thus, solves the output corridor control problem. It complements the results in \cite{MPZh23,MPZh23a} where the focus is on producing, via pulse-modulated feedback design, a given stable periodic dosing scheme.
In general, the presented control law constitutes an event-based controller of a simple structure for a positive  third-order LTI plant. Making use of the recent results in  \cite{PRM23}, it can be generalized to an arbitrary plant order. Since the IGO constitutes an example of nonlinear and non-smooth control, it is suitable for control of block-oriented nonlinear models where the (static) plant nonlinearity can be incorporated into the feedback modulation functions.

The rest of the paper is as follows. First, an example of a pharmaceutical application is briefly described to motivate the proposed control strategy (Section~\ref{sec:NMB}). Then, the mathematical model of the plant and the controller structure are defined, and the control problem at hand is formally stated 
(Section~\ref{sec:problem}). Its solution is presented in Section~\ref{sec:solution}.
Finally, a controller that keeps the measured output in the pharmaceutical application within a desired corridor is designed and studied in simulation (Section~\ref{sec:design}).

\section{Example of dosing application}\label{sec:NMB}
Neuromuscular blockade (NMB) is induced in surgery during
anaesthesia by administering muscle relaxant
agents to improve surgical conditions.  To minimize adverse side-effects  of the relaxants,  the minimum amount of drug is  administered to adequately paralyze the patient, which requires accurate objective quantification of the NMB depth.

A continuous Wiener model for NMB with the muscle relaxant {\it atracurium} under general  closed-loop  anesthesia  is introduced in \cite{SWM12}. The parameters of the model  are  estimated  from clinical data by means of a Particle Filter and an extended Kalman Filter  in \cite{RMM14}. 
The model input $u(t)$   is the administered atracurium rate in $\lbrack \mu\mathrm{g}\ \mathrm{kg}^{-1}\mathrm{min}^{-1} \rbrack $, positive and bounded, i.e. $0\le u(t) \le u_{\max}$. 
 The model output $y(t)$ $\lbrack \% \rbrack$ is the NMB level measured by a train-of-four monitor (peripheral nerve stimulator). Full recovery from NMB (zero drug concentration) corresponds then to $y(t)=100\%$. 

The linear model part consists of the transfer function 
\begin{equation}\label{eq:lin_NMB}
W(s)=\frac{\bar Y(s)}{U(s)}=\frac{v_1 v_2 v_3 \alpha^3}{(s+v_1\alpha)(s+v_2\alpha)(s+v_3\alpha)},
\end{equation}
where $\bar Y (s)$ is the Laplace transform of the  linear dynamic part 
output $\bar y(t)$ and $U(s)$ is the Laplace transform of the input.  The parameter $0<\alpha\le  0.1$ is patient-specific and estimated from data whereas the other transfer function parameters are fixed, $v_1=1$, $v_2=4$, and $v_3=10$. The NMB model output is related to the output of the transfer function by the nonlinear Hill-type function
\begin{equation}\label{eq:nonlin_NMB}
y(t)= \varphi(\bar y)\triangleq\frac{100 C_{50}^\gamma}{ C_{50}^\gamma + {\bar y}^\gamma(t)},
\end{equation}
where $C_{50}=3.2425$ $ \mu \mathrm{ g} \ \mathrm{ml}^{-1} $ is the drug concentration producing 50\% of the maximum effect and $0<\gamma\le 10$ is a patient-specific parameter.

In the beginning of a surgery, there is no muscle relaxant in the blood stream, i.e. $\bar y=0$ and $y(t)=100\%$. Then a larger bolus dose of {\it atracurium} (calculated as 400--500 $\mu \mathrm{ g}$ for a $\mathrm{kg}$ of patient weight) is administered to induce the state of NMB. When the desired NMB depth is reached, it is maintained by repeating a suitable drug dose each $15-25~\mathrm{min}$. In closed-loop anesthesia \cite{IMK15}, one thus distinguishes between induction phase and maintenance phase, with typically distinct control objectives. Notably, the instructions for the use of a muscle relaxant agent are written in terms of doses whereas a conventional feedback controller manipulates the flow of the drug. In the latter case, the medication is administered continuously over time and only the cumulative dose over a time interval can be monitored.

\section{Impulsive  Output Corridor Control}\label{sec:problem}

Motivated by the introductory example in Section~\ref{sec:NMB},   the control problem under consideration is formally stated below. Transfer function~\eqref{eq:lin_NMB} has a state-space realization
\begin{equation}                            \label{eq:state-space-NBM}
\dot{x}(t) =Ax(t)+Bu(t), \quad \bar y(t)=Cx(t),
\end{equation}
where $$A=\begin{bmatrix} -a_1 &0 &0 \\ g_1 & -a_2 &0 \\ 0 &g_2 &-a_3 \end{bmatrix}, B=\begin{bmatrix} 1 \\ 0 \\ 0\end{bmatrix}, C =\begin{bmatrix} 0 &0 &1 \end{bmatrix},
$$
$a_i=v_i\alpha>0$ are  positive distinct constants and $g_1,g_2>0$ are chosen to yield $g_1g_2=v_1v_2v_3\alpha^3$.

The design objective is to obtain an impulsive controller satisfying the following condition: the measured output of \eqref{eq:nonlin_NMB}
(after a transient period) is maintained in the pre-defined corridor $\lbrack y_{\min},y_{\max}\rbrack$, where $y_{\min}>0$.
To ensure that the output $y(t)$ does not leave this corridor for $t$ being large, we impose a formally stronger requirement
\begin{equation}\label{eq:range-open}
0< y_{\min}\le\liminf_{t\to\infty}y(t)\leq\limsup_{t\to\infty}y(t)\le y_{\max}.
\end{equation}


With  $\varphi$  defined by \eqref{eq:nonlin_NMB} as a monotonous decreasing function, let $\bar y_{\min}$ and $\bar y_{\max}$ be the solutions to the equations 
\[
y_{\max}=\varphi(\bar y_{\min}), \quad y_{\min}=\varphi(\bar y_{\max}),
\]
which do not generally require an analytical expression for $\varphi^{-1}(\cdot)$ and can be found numerically.
Then \eqref{eq:range-open} can be rewritten in the following equivalent form
\begin{equation}\label{eq:range-open1}
\bar y_{\min}\le\liminf_{t\to\infty}\bar y(t)
<\limsup_{t\to\infty}\bar y(t)\le \bar y_{\max}.
\end{equation}


\paragraph*{Impulsive controller} The impulsive nature of the sought controller is implied by the strategy commonly pursued in manual dosing applications, i.e. where a dose is administered and then the process response to it is monitored to assess the time and amount of the next dosing event. 
Using the Dirac $\delta$-function formalism, the controller is 
\[
u(t)=\sum_{n=0}^{\infty}\lambda_n\delta(t-t_n),
\]
where $0=t_0<t_1<\ldots$ is the sequence of discrete (not necessarily equidistant)  time instants, $t_n\to\infty$ and $\lambda_n$ are impulse weights corresponding to individual doses.
To avoid distributions (i.e. the $\delta$-functions), the closed-loop system is  recast as a hybrid system comprising  the continuous-time dynamics in \eqref{eq:state-space-NBM} for $t\in(t_n,t_{n+1})$
\begin{equation}                            \label{eq:1a}
\dot{x}(t) =Ax(t), \quad \bar y(t)=Cx(t),\quad\forall t\in(t_n,t_{n+1})
\end{equation}
subject to the instantaneous jumps at the instants $t_0,t_1,\ldots$ 
\begin{equation} \label{eq:1b}
x(t_{n}^+)=x(t_n^-)+\lambda_n B,\quad n=0,1,\ldots
\end{equation}
The minus and plus in a superscript in~\eqref{eq:1b} stand for the left-sided and
a right-sided limit, respectively.
Denoting $X_n\triangleq x(t_n^-)$, the  discrete-time dynamics of $X_n$ are given by
\begin{equation} \label{eq:1c}
X_{n+1}= \e^{(t_{n+1}-t_n)A}(X_n+\lambda_n B),\quad n=0,1,\ldots
\end{equation}
The knowledge of $X_n$ allows to uniquely recover the trajectory on the interval $(t_n,t_{n+1})$ via~\eqref{eq:state-space-NBM} and~\eqref{eq:1b}:
\begin{equation} \label{eq:1d}
x(t)=\e^{(t-t_n)A}(X_n+\lambda_n B),\quad t\in(t_n,t_{n+1}).
\end{equation}

To relate the timing and magnitude of the jumps to the measured continuous output, introduce the functions 
\begin{equation}\label{eq:modulation}
    T_n =\Phi(\bar y(t_n)), \quad \lambda_n=F(\bar y(t_n)), \quad T_n=t_{n+1}-t_n.
\end{equation}
In pulse-modulated control \cite{GC98}, $\Phi(\cdot)$ is termed the \emph{frequency modulation} function of the impulsive feedback, and 
$F(\cdot)$ is its \emph{amplitude modulation} function. 

Notably, the closed-loop pulse-modulated system \eqref{eq:1a}, \eqref{eq:1b}, \eqref{eq:modulation} with linear output $\bar y$ is identical with the Impulsive Goodwin Oscillator as introduced in \cite{MCS06}, \cite{Aut09} assuming that
the modulation functions $F$ and $\Phi$ are well-defined, continuous and monotonic on $[0,\infty)$, $F(\cdot)$ is non-increasing, $\Phi(\cdot)$ is non-decreasing, and
\begin{equation}                             \label{eq:2a}
0<\Phi_1\le \Phi(\cdot)\le\Phi_2, \quad 0<F_1\le F(\cdot)\le F_2,
\end{equation}
where $\Phi_1$, $\Phi_2$, $F_1$, $F_2$ are positive constants.

\paragraph*{1-cycle}  A periodic solution of \eqref{eq:1a}, \eqref{eq:1b}, \eqref{eq:modulation} is called 1-cycle if there is only one firing of the pulse-modulated feedback in the least period. Then, a 1-cycle is completely described by two parameters $\lambda, T$ where $\lambda_n=\lambda$, $T_n=T$ for all $n=0,1,\ldots$  From \eqref{eq:1c}, for a 1-cycle  with the initial condition $x(0^-)=X$, it applies 
\begin{align*}
    x(0^+)&=X+\lambda B, \\
    x(T^-)&= \e^{AT}x(0^+)=\e^{AT}(X+\lambda B).
\end{align*}
Now, since the solution is $T$-periodic, one has
\begin{equation}\label{eq:1-cycle}
    X=Q(X), \quad Q(\xi)\triangleq\mathrm{e}^{A\Phi(C\xi)}\left( \xi+ F(C\xi)B \right),
\end{equation}
that is, $X$ is the fixed point of the map $Q(\cdot)$ that completely characterizes the 1-cycle with the parameters $\Phi(CX)=T$, $F(CX)=\lambda$.

\paragraph*{Output corridor impulsive control problem} Now the controller design problem at hand can be formulated in the following way. Given nonlinear plant model \eqref{eq:state-space-NBM}, \eqref{eq:nonlin_NMB}, design the modulation functions $\Phi(\cdot)$ and $F(\cdot)$ calculating the sequences $\lambda_n$, $T_n$ from the continuous plant output $y(t)$ such that   inequality~\eqref{eq:range-open} (or, equivalently, \eqref{eq:range-open1}) is satisfied.

Notice that the minimal time interval condition $t_{n+1}-t_n\geq \Phi_1, \ n\in\mathbb{N}_0$ implies that the solution is well defined for $t\geq 0$ and no Zeno trajectories exist. This requirement is also natural in drug dosing applications, where enforcing a minimal interval between two consecutive doses is critical to patient safety. The upper bound on the interval between two doses is also imposed by \eqref{eq:2a}, i.e. $t_{n+1}-t_n\leq \Phi_2$. 

\section{Solution}\label{sec:solution}
The output corridor impulsive control problem formulated above does not explicitly specify what kind of dynamics are exhibited by the  closed-loop system. A logical and straightforward choice is to select the impulsive feedback so that it forces the system into a sustained periodic solution. The simplest instance of periodic solution in an impulsive system is a $1$-cycle where the impulses of the constant weight $\lambda_n=\lambda, \ n\in \mathbb{Z}$ occur at equidistant time instants, i.e. $T_n=T, \ n\in \mathbb{Z}$. In fact, other types of oscillations,  e.g. $n$-cycles, chaotic or quasiperiodic solutions, are also feasible but not considered here. Besides satisfying the output corridor condition in \eqref{eq:range-open}, the designed $1$-cycle has to be orbitally stable to enable convergence to the desired periodic solution under transitory output perturbation.

As seen from \eqref{eq:1-cycle}, a 1-cycle is defined by the fixed point $X$. Since it is also completely characterized by the cycle parameters $\lambda, T$, there is a way for calculating the fixed point from the plant and  cycle parameters. 
Following \cite{D03,DeBoor2005}, introduce the first divided difference of a function $f$ as a function of two variables 
\[
f[z_0,z_1]\triangleq \frac{f(z_1)-f(z_0)}{z_1-z_0},
\]
which expression is well defined if and only if $f(z_1)$, $f(z_0)$ exist and $z_0\ne z_1$. The second divided difference is a  function of three variables and is defined by
\[
f[z_0,z_1,z_2]\triangleq\frac{f[z_1,z_2]-f[z_0,z_1]}{z_2-z_0},
\]
where $f(z_0),f(z_1),f(z_2)$ exist and $z_0,z_1,z_2$ are pairwise different. Higher-order divided differences are then calculated in a recursive manner.

Denote, for brevity,
\[
\mu(z)\triangleq \frac{\e^z}{1-\e^z}, 
\quad z\ne 0.
\]

\begin{prop}{(\cite[Proposition~2]{MPZh23})}\label{th:fp}
    Given the parameters of $1$-cycle $T>0$, $\lambda>0$, the fixed point $ X= [x_{1} \ x_{2} \ x_{3} ]^\intercal, X>0 $ of the map $Q$ in~\eqref{eq:1-cycle} is calculated as
    \begin{equation} \label{eq:x0}
        X=\lambda \mu(AT) B,
    \end{equation}
    or, in terms of individual elements
\begin{align}\label{eq:x0_elements}
     x_{1}&=\lambda\mu(-a_1T), \\
     x_{2}&=\lambda g_1T\mu[-a_1T,-a_2T],\nonumber\\
     x_{3}&=\lambda g_1g_2T^2\mu[-a_1T,-a_2T,-a_3T]. \nonumber
 \end{align}
\end{prop}
The closed-form expressions for the fixed point coordinates imply its uniqueness.


Now the relationship between the 1-cycle parameters and the output corridor limits can be established.
\begin{prop}\label{th:corridor}
Let closed-loop system  \eqref{eq:1a}, \eqref{eq:1d}, \eqref{eq:modulation} evolve in the 1-cycle corresponding to a fixed point $X$ defined by Proposition~\ref{th:fp}.
Let $\tau_1<\tau_2<\ldots<\tau_k$ denote all roots\footnote{Since the left-hand side of~\eqref{eq:roots_dy} is a real-analytic function on $\mathbb{R}$, it has only finite number of roots on every closed interval, e.g., on $[0,T]$.} of the equation
\begin{equation}\label{eq:roots_dy}
  C\e^{A\tau}{(I-\e^{AT})}^{-1}AB=0
\end{equation}
on the interval $(0,T)$. 
Then the system output satisfies inequalities \eqref{eq:range-open1} with 
\begin{align*}
    \bar y_{\max}&=\lambda\max_{i=1,k} C\e^{A\tau_i}{(I-\e^{AT})}^{-1}B,\\
    0<\bar y_{\min}&=\lambda\min_{i=1,k} C\e^{A\tau_i}{(I-\e^{AT})}^{-1}B<x_3.
\end{align*}
\end{prop}
\begin{proof} Consider the solution $x(t)$ on the interval $\lbrack 0, T^-\rbrack$.  With the initial condition $x(0^-)=X$, \eqref{eq:1b} gives
\[
x(0^+)=X+ \lambda B.
\]
Using \eqref{eq:1c}, after the jump, the state vector is governed for $t\in \lbrack 0^+, T^-\rbrack$ by
\begin{align*}
    x(t)&= \e^{At}(X+\lambda B)= \lambda \e^{At} \left( {\e^{AT}(I-\e^{AT})}^{-1}+I \right)B\\
    &= \lambda  \e^{At} {(I-\e^{AT})}^{-1}B,
\end{align*}
where the expression for the fixed point in \eqref{eq:x0} is utilized. 

Notice that $y(t)$ is continuous at all $t$ is continuous despite the jumps in the state vector~\eqref{eq:1b}, because $CB=0$. Also, $y(0)=y(T)$ in view of the $T$-periodicity. Hence, $y$ attains its minimal and maximal values on $[0,T]$ in view of the Weierstrass theorem.
We will show that neither the minimum nor the maximum can be attained at the endpoints $t=0$ and $t=T$.
Indeed, 
\[
\bar y(t)=C \e^{At}(X+ \lambda B)>0
\]
where $X>0$, $\lambda>0$, and $A$ is Metzler. The output $\bar y(t)$ is thus decomposed into the sum 
\begin{align*}
    \bar y(t)&=\bar y_1(t)+\bar y_2(t), \\
    \bar y_1(t)&=C \e^{At}X>0, \\
    \bar y_2(t)&= C \e^{At}B\ge 0.
\end{align*}
Consider the derivative of the initial conditions response 
    \[
    \dot {\bar y}_1(t)=C\e^{A t}AX.
    \]
 According to \cite[Proposition~3]{MPZh23}, it holds that $AX<0$, where the inequality is understood element-wise. Being defined by \eqref{eq:state-space-NBM}, the matrix $A$ is Metzler and then the exponential matrix maps a negative vector to a negative vector $\forall t$, i.e.  $\dot {\bar y}_1(t)<0, t\in \lbrack 0, T \rbrack$ and ${\bar y}_1(t)$ is monotonously decreasing over a period of the solution $x(t)$.
The derivative of the impulse response 
 \[
 \dot {\bar y}_1(t)=C\e^{A t}AB
 \]
 is not sign-definite because the vector $AB$ includes positive, negative, and zero elements. From the matrices of \eqref{eq:state-space-NBM}, one has ${\bar y}_1(0)=\dot {\bar y}_1(0)=0$. Thus $\dot {\bar y}(0)<0$ and the output $\bar y(t)$ is  decreasing from its initial condition $\bar y(0)=x_3$, which cannot be the minimal value of $\bar y$ on $[0,T]$. Similarly, one may notice that $x(t)=\e^{(t-T)A}x(T^-)=\e^{(t-T)A}X$ for $t\in(0,T)$, and hence $\dot{\bar y}(t)=C\e^{(t-T)A}AX\to CAX<0$ as $t\to T^-$, in other words, $\bar y(t)>\bar y(T)$ for $t\in(0,T)$ being close to $T$, and thus $\bar y(0)=\bar y(T)$ is also not the maximum.

We have shown that both the minimal and the maximal value of $\bar y$ are attained at $t\in (0,T)$. By noticing that
\[
\dot x(t)= \lambda  \e^{At} {(I-\e^{AT})}^{-1}AB\;\;\forall t\in(0,T),
\]
one proves that to achieve an extremum at a time instant $\tau$, the derivative of the output has to satisfy
\[
\dot{\bar y}(t)|_{t=\tau}=C\dot x(t)= \lambda  C\e^{A\tau} {(I-\e^{AT})}^{-1}AB=0,
\]
or, since $\lambda\ne 0$, one arrives to \eqref{eq:roots_dy}. The extreme output values are then
\[
\bar y(\tau_i)=\lambda\max_{i=1,k} C\e^{A\tau_i}{(I-\e^{AT})}^{-1}B, \quad i=1,\dots,k.
\]
\end{proof}
\subsection{Design} 
Proposition~\ref{th:corridor} specifies the output corridor for closed-loop impulsive feedback system  \eqref{eq:1c}, \eqref{eq:1d}, \eqref{eq:modulation}  whose parameters are known and that exhibits a 1-cycle with certain parameters. 
Now consider an algorithm that solves  a converse problem  formulated in Section~\ref{sec:problem}, where the impulsive feedback described by \eqref{eq:1d}, \eqref{eq:modulation}  keeps the output of the continuous model in \eqref{eq:lin_NMB} within a corridor given by \eqref{eq:range-open}.
\paragraph*{Algorithm~1}
\begin{description}
    \item[Step 1:]\label{itm:S1} \  Define the parameters of \eqref{eq:lin_NMB} and the  desired output corridor $[\bar y_{\min}, \bar y_{\max}]$.
\item[Step 2:]\label{itm:S2} \  Select a suitable interval for the period ${\cal T}=[T_{\min}, T_{\max}]$. By gridding over $T_j\in {\cal T}$, calculate
\begin{equation}\label{eq:roots_dy_T_j}
  \tau_{i,j} : \quad  C\e^{A\tau}{(I-\e^{AT_j})}^{-1}AB |_{\tau=\tau_i}=0,
\end{equation}
and evaluate 
\begin{align*}
    z^{(j)}_{\max}&=\max_{i=1,k} C\e^{A\tau_{i,j}}{(I-\e^{AT_j})}^{-1}B,\\
    z^{(j)}_{\min}&=\min_{i=1,k} C\e^{A\tau_{i,j}}{(I-\e^{AT_j})}^{-1}B.
\end{align*}
\item[Step 3:]\label{itm:S3} \  Obtain the period of the 1-cycle $T$ as $T=T_k$ by solving 
\begin{equation}\label{eq:design_T}
  k= \arg\min_{j} \left| \frac{\bar y_{\max}}{\bar y_{\max}- \bar y_{\min}}-\frac{z^{(j)}_{\max}}{z^{(j)}_{\max}-z^{(j)}_{\min}}\right| 
\end{equation}
The ratios in  \eqref{eq:design_T} are independent of $\lambda$ due to the linearity of \eqref{eq:1a}.  

\item[Step 4:]\label{itm:S3} \  With the value of $T$ obtained in Step~3, calculate the impulse weight of the  1-cycle 
\begin{equation}\label{eq:lambda}
    \lambda= \frac{\bar y_{\max}-\bar y_{\min}}{z^{(k)}_{\max}-z^{(k)}_{\min}}.
\end{equation}
\item[Step 5:]\label{itm:S3} \ Evaluate the modulation functions $\Phi(\cdot)$ and $F(\cdot)$ rendering the desired orbitally stable 1-cycle by applying the approaches described in either \cite{MPZh23} or \cite{MPZh23a}. The implementation of this step hinges on parametrization of the modulation functions and does not necessarily have a unique solution.
\end{description}

To characterize transient solutions of the IGO, stability of the 1-cycle has to be established.
\begin{prop}[Stability]\label{th:stability} The periodic solution satisfying the conditions of Proposition~\ref{th:corridor} is orbitally stable if and only if the Jacobian of the map $Q(\cdot)$
\begin{equation}\label{eq:closed_loop}
    Q^\prime(X)= \e^{AT}+KC,
\end{equation}
where
\begin{equation}\label{eq:gain}
    K= \begin{bmatrix}
    J &D
\end{bmatrix}\begin{bmatrix}
 F^\prime(CX)  \\  \Phi^\prime(\bar y(CX) 
\end{bmatrix}, J=\e^{AT}B, D=AX
\end{equation}
    is Schur-stable. If, in addition, the initial conditions to \eqref{eq:1a} are within the basin of attraction of the fixed point $X$ and all the eigenvalues of $Q^\prime(X)$ are positive, then the convergence of the sequence $\bar y(t_k)$ to the 1-cycle is monotonous.
\end{prop}
\begin{proof}
    The characterization of Jacobian \eqref{eq:closed_loop} as \eqref{eq:gain} follows from  \cite[Proposition~3]{MPZh23}. Orbital stability of 1-cycle implies that the orbit possesses a basin of attraction.  The eigenvalues of $Q^\prime(X)$ are the characteristic multipliers and their positivity yields monotonous convergence of $\bar y(t_k)$ to $CX$, see \cite{C99}.
\end{proof}
The derivatives $F^\prime$ and $\Phi^\prime$ exist since the modulation functions are assumed to be continuous. When $F^\prime(CX)=\Phi^\prime(CX)=0$, orbital stability follows trivially due to $A$ being Hurwitz. 

 In view of the dosing application described in Section~\ref{sec:NMB}, the basin of attraction of the designed 1-cycle has to include the point $x=0$ since it designates the starting point of the NMB procedure.

The expression in \eqref{eq:closed_loop} is exactly the same as what describes the closed-loop dynamics in static output feedback design \cite{SAD97} for LTI systems. In the IGO, the role of control gains is although played by the slopes of the modulation functions, i.e. $F^\prime(CX)$ and  $\Phi^\prime(CX)$. Clearly, the case of constant modulation functions corresponds to zero gain. It is instructive to note that both the continuous part of the IGO (the plant) and the discrete part of it (the pulse-modulated controller) feature only positive signals. Yet, since $F(\cdot)$ is  non-increasing, $\Phi(\cdot)$ is non-decreasing, $J>0$, and $D<0$ (see \cite[Proposition~3]{MPZh23}), then
\[ JF^\prime(\cdot)+D\Phi^\prime(\cdot)\le 0, \]
and the feedback is negative, i.e. it enforces faster convergence to the stationary solution (1-cycle) and decreases the sensitivity of the closed-loop system to model uncertainty. The control law implemented by the IGO can therefore be described as {\it positively-valued negative feedback}.

\subsection{Nonlinear plant}
Nonlinear models are frequent in process control and biomedical systems. The IGO possesses highly nonlinear dynamics \cite{ZCM12b} due to the use of pulse-modulated feedback even when the continuous controlled plant (cf. \eqref{eq:1a}) is linear. This opens up for generalizing the IGO to continuous plants with static nonlinearities in the control signals or/and measurements.

Consider a Hammerstein model

\begin{equation}\label{eq:hammerstein}
    \dot x=Ax+B\varphi_h(u), \quad y=Cx,
\end{equation}
where $\varphi_h(\cdot)$ is a given positive continuous function. In context of the IGO, when \eqref{eq:hammerstein} is the plant,  the pulse-modulated feedback law becomes
\begin{equation}                             \label{eq:2a}
x(t_n^+) = x(t_n^-) +\varphi_h(\lambda_n) B,                                         
\end{equation}
whereas \eqref{eq:1a} can be kept intact. Let $\bar \lambda_n=\varphi_h(\lambda_n)$. Then, the case of Hammerstein model can be reduced to that already considered in Section~\ref{sec:problem}, i.e. \eqref{eq:1a}, \eqref{eq:modulation}, where the modulation functions are modified to be
\[
T_n =\Phi(y(t_n)), \quad \bar\lambda_n=F(y(t_n)).
\]
The weight $\lambda_n$ has now to solve 
\begin{equation}\label{eq:hammerstein_amplitude}
    \varphi_h(\lambda_n)= F(y(t_n)).
\end{equation}
When an inverse of $\varphi_h(\cdot)$ is available then
\[
\lambda_n= \varphi^{-1}( F(y(t_n))= (\varphi^{-1}\circ F)(t_n),
\]
where $\circ$ is composition of two functions. Otherwise, equation \eqref{eq:hammerstein_amplitude} is to be solved numerically for each value $y(t_n)$. Calculating the controller output via a numerical evaluation of the roots of an algebraic equation is common practice in pulse-modulated systems \cite{GC98}. Notably, the presence of a static input nonlinearity in the plant does not impact the frequency modulation mechanism, i.e $\Phi(\cdot)$.

Consider now a Wiener model
\begin{equation}\label{eq:wiener}
    \dot x=Ax+Bu, \quad y=\varphi_w(\bar y), \quad \bar y=Cx,
\end{equation}
where  $\varphi_w(\cdot)$ is a positive continuous function. Given an impulsive input, the continuous plant is already in the form of \eqref{eq:1a}.  The feedback law also preserves its form but the modulation functions become
\[
T_n =(\Phi\circ \varphi_w )(\bar y(t_n)), \quad \lambda_n=(F\circ \varphi_w )(\bar y(t_n)).
\]

\color{black}

\section{Output corridor control of NMB}\label{sec:design}
%
In this section, Algorithm~1 is applied step-by-step to design impulsive feedback for the NMB model described in Section~\ref{sec:NMB}.
\paragraph*{Step~1} Consider model \eqref{eq:lin_NMB},  \eqref{eq:nonlin_NMB} where the individualization parameters are set to the mean population values $\alpha=0.0374$, $\gamma=2.6677$, and $C_{50}=3.2425$.  The elements of the state matrix $A$ in \eqref{eq:1a} are then $a_1=v_1\alpha$, $a_2=v_2\alpha$, $a_3=v_3\alpha$, $g_1=v_1\alpha$, and  $g_2=v_2 v_3 \alpha^2$. Impulse responses of the model to impulses of different amplitudes are shown in Fig.~\ref{fig:impulse} to illustrate the nonlinear dynamics.

From the clinical data from \cite[Fig.~4]{SWM12}, the NMB depth is to be kept within the range $2\% \le y(t) \le 10\%$ throughout the surgery. 
The nonlinear function in \eqref{eq:nonlin_NMB} is invertable
\[
\bar y(t)=\varphi^{-1}(y(t))= C_{50} {\left({\frac{100}{y(t)}-1}\right)}^{\frac{1}{\gamma}},
\]
and the desired interval of $y(t)$ is equivalent to 
\[
\bar y_{\min}=7.3889 \le \bar y(t) \le  13.9463=\bar y_{\max}.
\] 
\paragraph*{Step~2} Select ${\cal T}=[15, 45]$. The ratio 
${z^{(j)}_{\max}}/{(z^{(j)}_{\max}-z^{(j)}_{\min})}$
is plotted as the  function of $T_j$ in Fig.~\ref{fig:T} (blue curve).

\paragraph*{Step~3} 
Solving \eqref{eq:design_T} numerically (see Fig.~\ref{fig:T}) gives $T=37.3834$.
\paragraph*{Step~4} 
The maintenance dose necessary to elevate the drug concentration from $\bar y_{\min}$ to $\bar y_{\max}$ is given by \eqref{eq:lambda}, i.e. $\lambda= 415.8412$. An equivalent way of defining the designed 1-cycle is the fixed point of Proposition~\ref{th:fp}
 $X^\intercal=\begin{bmatrix}  136.4461 &44.9637 & 7.4309\end{bmatrix}$.  The open-loop impulse response of the plant with the extrema is depicted in Fig.~\ref{fig:output_period}. 
 \paragraph*{Step~5} 
 
The modulation functions are subject to 
\begin{equation}\label{eq:lambda_T}
    F(\bar y_0)=\lambda, \quad \Phi(\bar y_0)=T, \quad \bar y_0=CX.
\end{equation}
Since $\varphi(\cdot)$ is (monotonously) decreasing, in contrast with the case of a LTI plant in Section~\ref{sec:solution}, $F(\cdot)$ has to be increasing and $\Phi(\cdot)$ has to be decreasing for  Wiener NMB model \eqref{eq:lin_NMB}, \eqref{eq:nonlin_NMB}. Indeed, when the NMB level $y(t)$ climbs too high, higher drug doses have to be administered more often.

Let the modulation functions be selected as
\[
F(\xi)\triangleq (\bar F \circ \varphi)(\xi), \quad \Phi(\xi)\triangleq (\bar \Phi \circ \varphi)(\xi),
\]
where $\bar F (\cdot), \bar \Phi (\cdot)$ represent the design degrees of freedom of the IGO and have to guarantee the desired 1-cycle in the closed-loop system as well as its (orbital) stability. Select these modulation functions as piecewise affine, i.e.
\begin{align*}
    \bar \Phi (\xi)= \begin{cases} \Phi_2 &   \Phi_2 < k_2\xi +k_1, \\
     k_2\xi +k_1 & \Phi_1 \le  k_2\xi +k_1 \le \Phi_2, \\
    \Phi_1  &  k_2\xi +k_1 < \Phi_1, 
     \end{cases}
\end{align*}
\begin{align*}
    \bar F (\xi)= \begin{cases} F_1 &  k_4\xi +k_3< F_1, \\
     k_4\xi +k_3 & F_1 \le k_4\xi +k_3 \le F_2, \\
    F_2 & F_2 <k_4\xi +k_3.
     \end{cases}
\end{align*}
From the bounds on the modulation functions, it follows that the feedback cannot administer a dose that is greater than $F_2$ or less than $F_1$. Further, no dose is administered sooner than $\Phi_1$ from the previous one and at least one dose is administered within a time interval of $\Phi_1$. These bounds are easily established from the available manual medication protocols in general anesthesia.



\begin{figure}[ht]
\centering 
\includegraphics[width=0.8\linewidth]{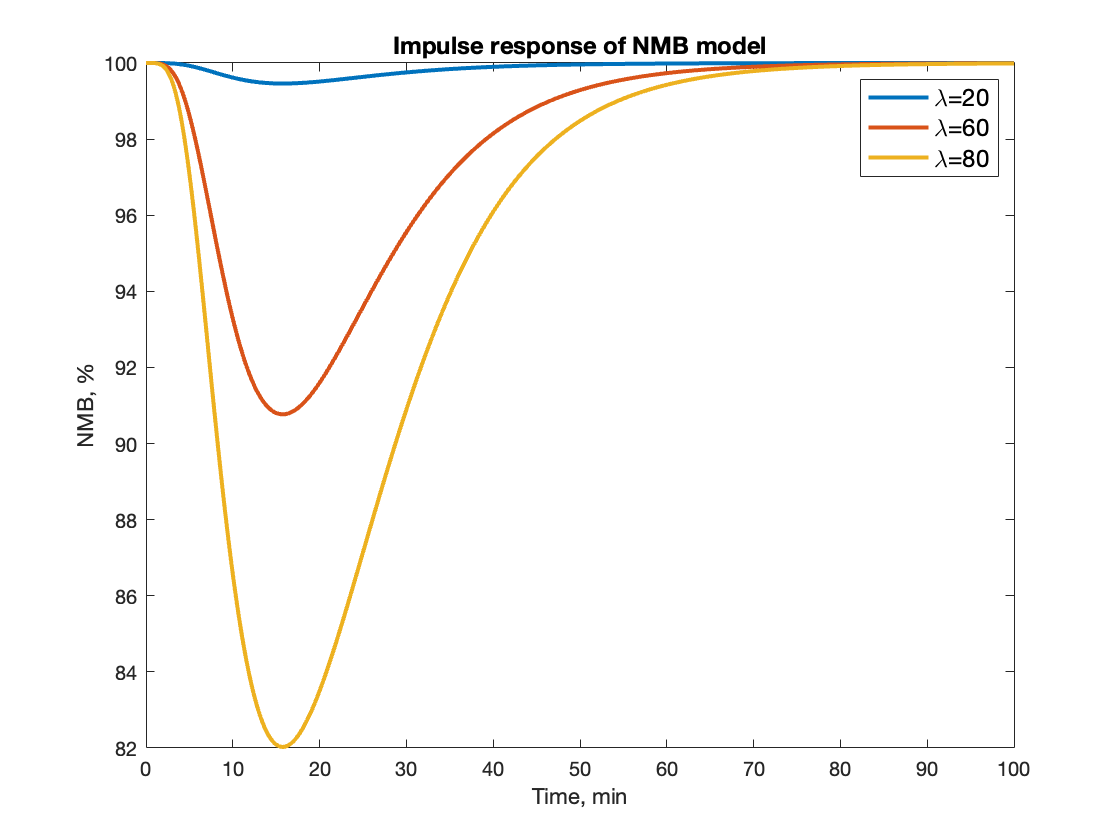}
\caption{Impulse response of NMB model \eqref{eq:lin_NMB}, \eqref{eq:nonlin_NMB} to $\lambda \delta(t)$, where $\delta(t)$ is Dirac delta. A dose of $\lambda$ causes a deeper NMB and the muscle relaxing  effect subsides with time due to drug elimination.}\label{fig:impulse}
\end{figure}

\begin{figure}[ht]
\centering 
\includegraphics[width=0.8\linewidth]{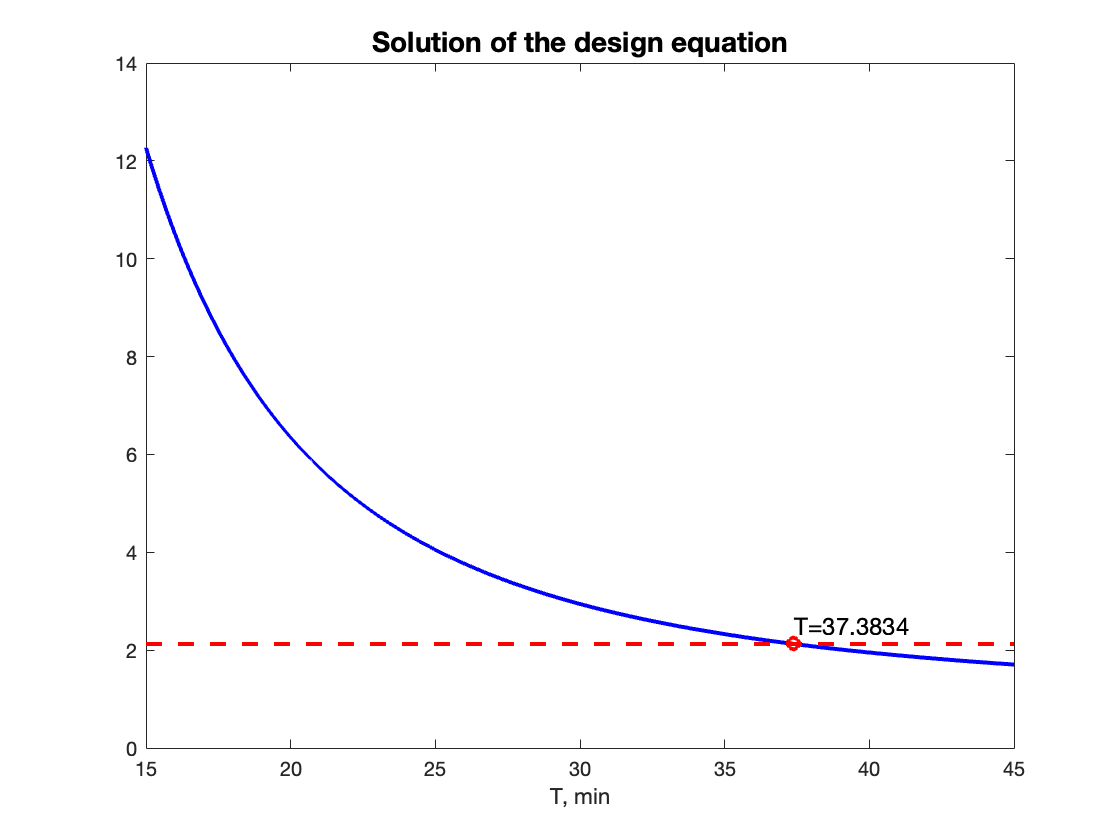}
\caption{Solution to equation \eqref{eq:design_T} for the assumed numerical values. The desired ratio ${\bar y_{\max}}/{(\bar y_{\max}-\bar y_{\min})}$  is in red. The function ${z^{(j)}_{\max}}/{(z^{(j)}_{\max}-z^{(j)}_{\min})}$ is in blue. The solution is $T=37.3834$.}\label{fig:T}
\end{figure}

\begin{figure}[ht]
\centering 
\includegraphics[width=0.75\linewidth]{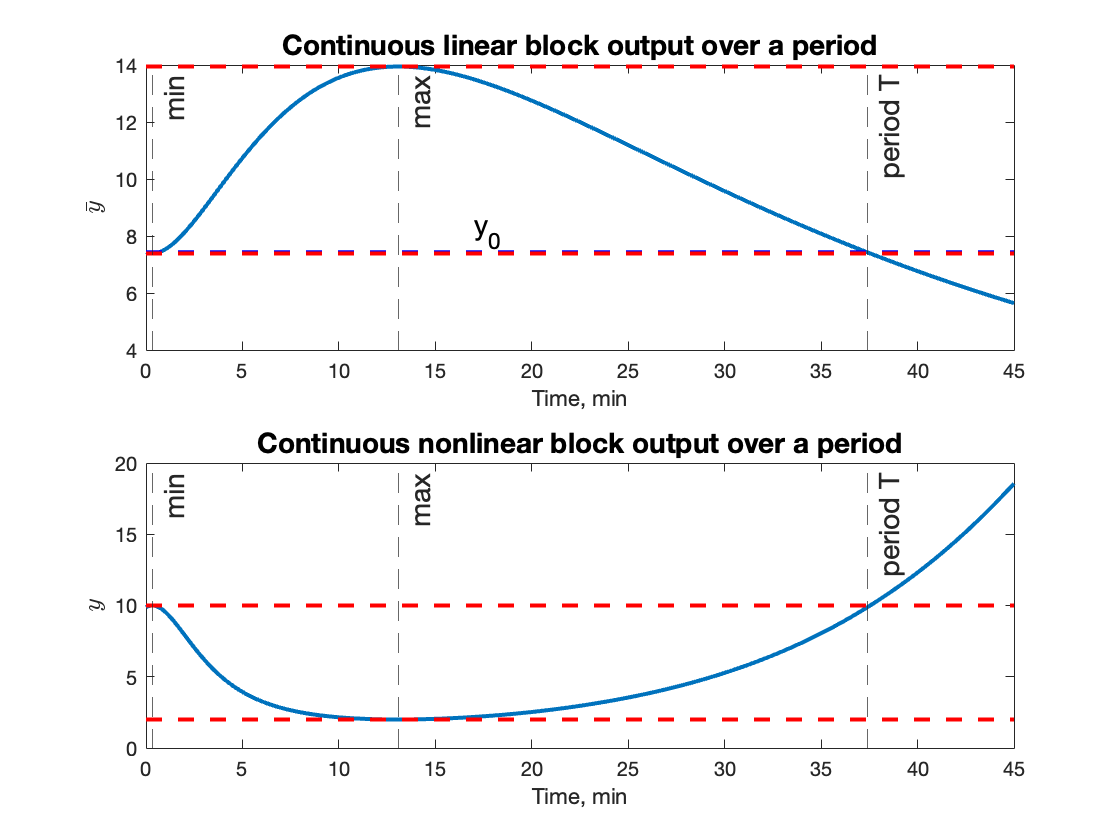}
\caption{Upper plot: The output $\bar y(t)$ in open loop in response to the impulsive input $\lambda\delta(t)$ (blue) and the corridor bounds $\bar y_{\max}, \bar y_{\min}$ (dashed red). The initial condition $\bar y_0=x_3$ (dashed blue) is slightly over $\bar y_{\min}$. Lower plot: the output $y(t)$ (blue) with the corresponding bounds $y_{\max},y_{\min}$ (dashed red) }\label{fig:output_period} 
\end{figure}

\begin{figure}[ht]
\centering 
\includegraphics[width=0.7\linewidth]{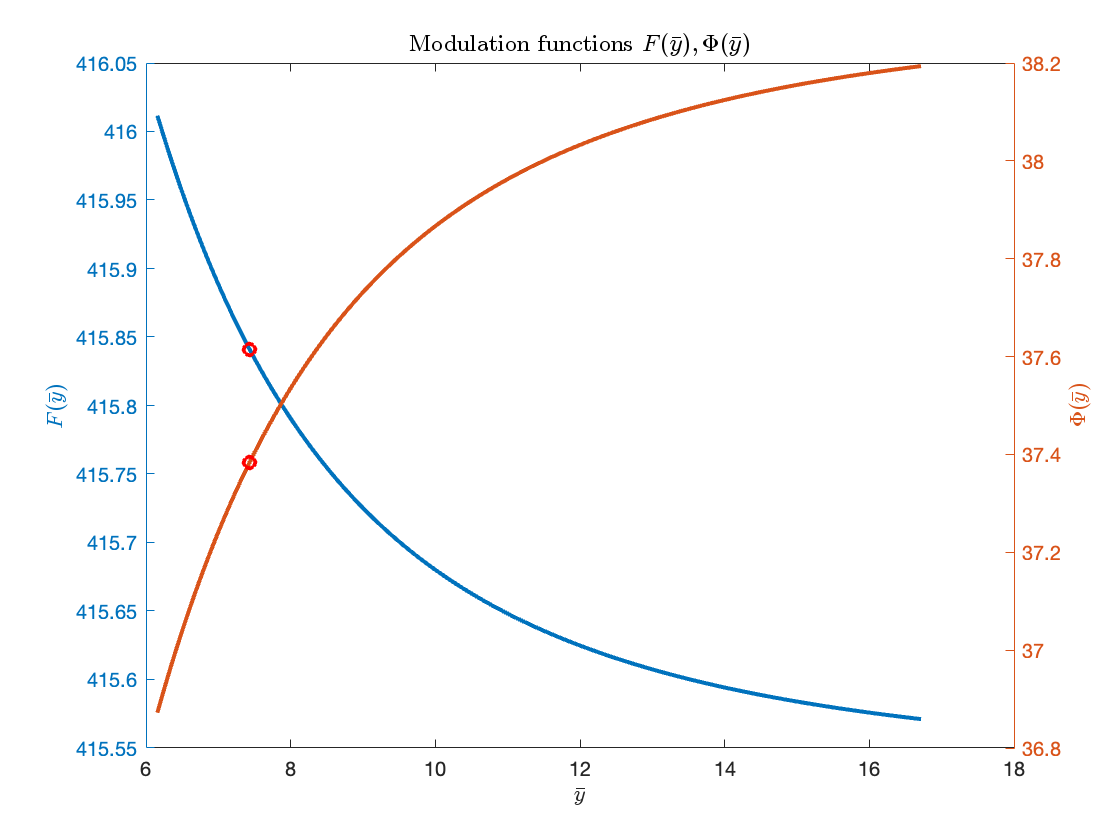}
\caption{The modulation functions $F(\bar y)$ (blue) and $\Phi(\bar y) $ (red). The cycle parameters $F(\bar y_0)=\lambda, \Phi(\bar y_0)=T$ are marked by red dot.}\label{fig:F_Phi}
\end{figure}
\paragraph*{Stability} To guarantee orbital stability of the designed 1-cycle, the slopes of the modulation functions have to satisfy the conditions of Proposition~\ref{th:stability}. By applying the chain rule and assuming that $F^\prime(\cdot)$ and $\Phi^\prime(\cdot)$ do not reach saturation
\begin{align*}
     F^\prime(\bar y_0)&=\bar F^\prime(\bar y_0)\varphi^\prime(\bar y_0)= k_4 \varphi^\prime(\bar y_0),\\ 
     \Phi^\prime(\bar y_0)&= \bar \Phi^\prime(\bar y_0)\varphi^\prime(\bar y_0)= k_2 \varphi^\prime(\bar y_0),
\end{align*}
where
\[
\varphi^\prime(\xi)= -\frac{\gamma 100 C_{50}^\gamma  \xi^{\gamma-1}}{ {(C_{50}^\gamma + {\xi}^\gamma)}^2}.
\]
According to Proposition~\ref{th:stability}, orbital stability of the designed 1-cycle is guaranteed by the eigenvalues of \eqref{eq:closed_loop} being within unit circle, where
\begin{align}\label{eq:gain_NMB}
K&= 
 \varphi^\prime(\bar y_{0}) \begin{bmatrix}
    J &D
\end{bmatrix}\begin{bmatrix}
 \bar F^\prime(\varphi(\bar y_{0}))  \\  \bar \Phi^\prime(\varphi(\bar y_{0})) \end{bmatrix}\nonumber\\
 &= \varphi^\prime(\bar y_{0}) \begin{bmatrix}
    J &D
\end{bmatrix} \begin{bmatrix} k_4 \\ k_2
\end{bmatrix},
\end{align}
and
\[
\varphi^\prime(\bar y_{0})=  -3.1921.
\]
Choosing $k_4=0.0313, k_2= -0.0940$ renders the eigenvalue spectrum of $Q(X)$ 
\[
\sigma \left( Q(X) \right)= \{ 0.1575, \ 0.0130, 3.5206\cdot 10^{-7}\},
\]
and the designed 1-cycle is orbitally stable as the spectral radius of the Jacobian is $\rho(Q(X) )=0.1575$. The feedback in \eqref{eq:gain_NMB} improves the convergence to the desired periodic solution compared to an open-loop mode 
since 
\[
\sigma \left( \e^{AT} \right)= \{0.2471, \  0.0037, \ 8.4715\cdot 10^{-7} \}.
\]
\paragraph*{Modulation at fixed point}
For the adopted parametrization of the modulation functions, 
\begin{align*}
    F(\bar y_{0})&= (\bar F \circ \varphi)(\bar y_{0})=\bar F ( \varphi(\bar y_{0}))=k_4 \varphi(\bar y_{0})+k_3= \lambda,\\
    \Phi(\bar y_{0})&= (\bar \Phi \circ \varphi)(\bar y_{0})= \bar \Phi (\varphi(\bar y_{0}))= k_2 \varphi(\bar y_{0})+k_1=T.
\end{align*}
The values $k_3=415.5321, k_1=38.3105$ are obtained by solving the equations above. The upper and lower bounds of the modulation functions are selected as $F_1=200, F_2=5000, \Phi_1=5, \Phi_2=45$. The resulting modulation functions $F(y),\Phi(y)$ are depicted in Fig.~\ref{fig:F_Phi}.
\begin{figure}[ht]
\centering 
\includegraphics[width=0.8\linewidth]{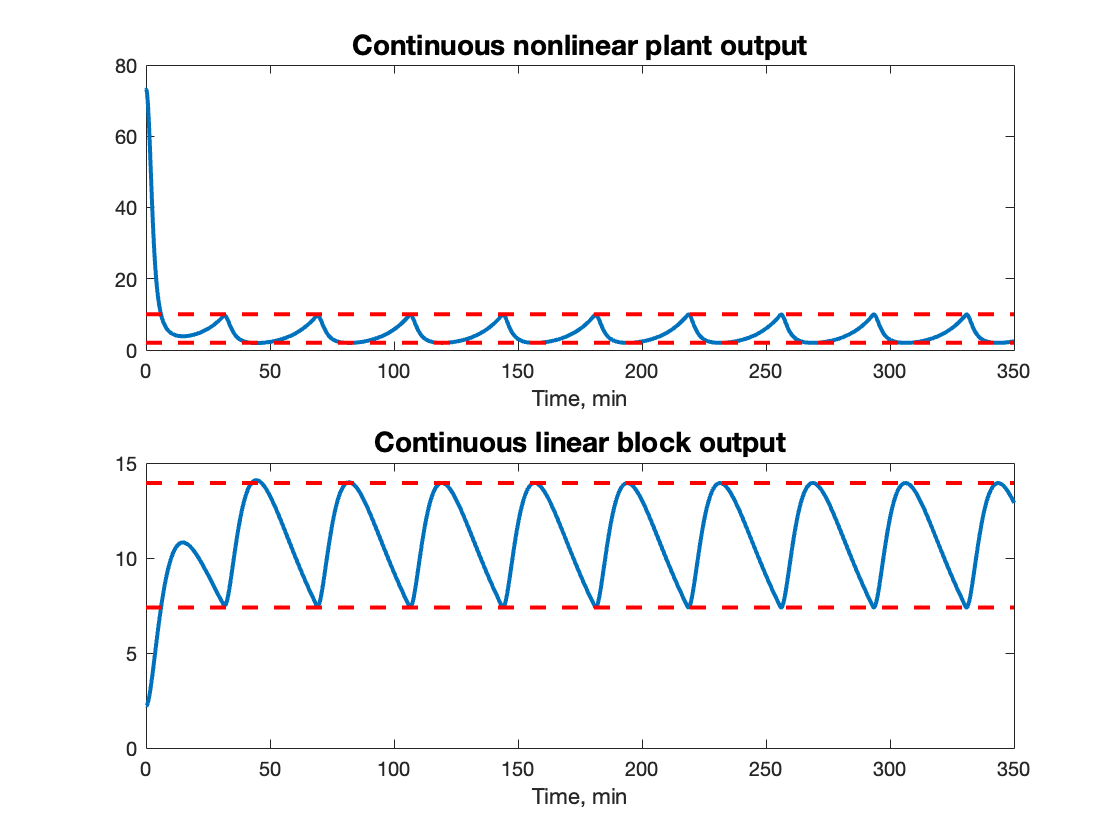}
\caption{Convergence to the designed cycle from an initial condition outside of the orbit (undergoes).  The sequence $y(t_n)$ is monotone, cf. $\sigma \left( Q(X) \right)$.}\label{fig:1-cycle}
\end{figure}

\section{CONCLUSIONS}
The problem of controlling the output of a positive Wiener or Hammerstein system to a pre-defined interval of values is considered. It is solved for a third-order time-invariant system by designing a pulse-modulated feedback that renders an orbitally stable periodic solution of a certain type in the closed-loop system. The character of the convergence from a feasible initial condition to the periodic solution is controlled by the slopes of the frequency and amplitude modulation functions. The resulting closed-loop dynamics are identical to those of the impulsive Goodwin's oscillator in 1-cycle. The proposed control approach is illustrated by simulation on a feedback drug dosing application in neuromuscular blockade.

\bibliography{refs,observer}
\bibliographystyle{ieeetr}

\end{document}